\documentclass[12pt]{amsart}

\usepackage{amsmath,amsfonts,amssymb,amscd,amsthm,amsbsy,graphics,epsf}
\usepackage{comment,graphicx}

\theoremstyle{plain}

\newtheorem{atheorem}{Theorem}

\theoremstyle{definition}

\numberwithin{equation}{section}

\begin{document}

\title []{Pluripolarity of graphs of quasianalytic functions in the sense of Gonchar}

\author []{Sevdiyor Imomkulov and Zafar Ibragimov}


\urladdr{}

\date{}

\thanks{}



\maketitle


Let $f$  be a continuous real valued function, defined on a segment $\Delta=[a,b]$  of the real axis $\mathbb R$ and let $\rho_n(f)$ be the deviation from the best approximation of $f$  on $\Delta$ by rational functions $r_n(x)$ of degree less than or equal to $n$. That is, 
$$
\rho_n(f)=\inf_{r_n}\max_{x\in\Delta}|f(x)-r_n(x)|,
$$ 
where the infimum is taken over all rational functions of the form
$$
r_n(x)=\frac {a_0x^n+a_1x^{n-1}+\dots +a_n}{b_0x^n+b_1x^{n-1}+\dots +b_n}.
$$
As usual, we denote by $e_n(f)$ the least deviation of $f$ on$\Delta$ from its polynomial approximation of degree less than or equal 
to $n$. Obvious, $\rho_n(f)\leq e_n(f)$ for every $n=0,1,2,\dots$. In papers (\cite{G1},\cite{G2}) Gonchar proved that the class of functions 
$$
R(\Delta)=\{f\in C(\Delta)\colon\lim_{\overline{n\to\infty}}\sqrt[n]{\rho_n(f)}<1\}
$$
possesses one of the important properties of the class of analytic functions. Namely, if 
$$
\lim_{\overline{n\to\infty}}\sqrt[n]{\rho_n(f)}<1
$$
and $f(x)=0$ on a set $E\subset\Delta$ of positive logarithmic capacity, then $f(x)\equiv 0$ on $\Delta$.

By analogy with the class 
$$
B(\Delta)=\{f\in C(\Delta)\colon\lim_{\overline{n\to\infty}}\sqrt[n]{e_n(f)}<1\},
$$
which is called the class of quasianalytic functions of Bernstein (see \cite{CLP}), we call $R(\Delta)$ the class of quasianalytic functions of Gonchar.

It is known that functions that are analytic on $\Delta$ are characterized by condition 
$$
{\overline\lim}_{n\to\infty}\sqrt[n]{e_n(f)}<1,
$$
according to Bernstein's theorem. It follows that the class $A(\Delta)$ of analytic functions on $\Delta$  is a subclass of $B(\Delta)\subset R(\Delta)$, i.e., $A(\Delta)\subset B(\Delta)\subset R(\Delta)$. It is clear that $A(\Delta)\neq B(\Delta)$. It was shown in (\cite{G1}) that there exist functions for which the rate of approximation by polynomials $e_n(f)$ tends to zero as slow as possible whereas $\rho_n(f)$ tends to zero as fast as possible. In particular, it follows that    
$$
B(\Delta)\neq R(\Delta).
$$

It is not hard to see that if $f\in A(\Delta)$, then its graph 
$$
\Gamma_f=\{(x,f(x))\in \mathbb C^2\colon x\in\Delta\}
$$
is a pluripolar set in $\mathbb C^2$. In (\cite{DF}) the authors constructed examples of smooth (infinitely differentiable) functions whose grath are not pluripolar in $\mathbb C^2$. In their recent paper (\cite{CLP}) the authors have proved that if $f\in B(\Delta)$, then its gaph $\Gamma_f$ is pluripolar in $\mathbb C^2$.

In this paper we prove following more general theorem.

\begin{atheorem} If $f\in R(\Delta)$, then its graph $\Gamma_f$ is pluripolar in $\mathbb C^2$.
\end{atheorem}\begin{proof}
By the hypothesis of the theorem there exists a sequence of natural numbers $n_k$  and a corresponding sequence of rational functions 
$$
r_{n_k}=\frac {p_{n_k}}{q_{n_k}}
$$
such that
$$
\rho_{n_k}(f)=\Big\Arrowvert f-r_{n_k}\Big\Arrowvert_{\Delta}\leq\alpha^{n_k},
$$ 
where $\alpha\in (0,1)$ is some fixed number. Without loss of generality we can assume that 
$$
\big\Arrowvert f\big\Arrowvert_\Delta\leq\frac {1}{2},\qquad\big\Arrowvert p_{n_k}\big\Arrowvert_\Delta\leq 1\qquad\text{and}\qquad \big\Arrowvert q_{n_k}\big\Arrowvert_\Delta=1.
$$
According to Bernstein-Walsh inequality (see, \cite{S2}) we have
$$
\big|p_{n_k}(z)\big|\leq e^{n_kV^\star(z,\Delta)}\qquad\text{and}\qquad \big|q_{n_k}(z)\big|\leq e^{n_kV^\star(z,\Delta)}
$$   
for any $z\in\mathbb C$ and $k\in N$. Here  
$$
V^\star(z,\Delta)={\overline\lim}_{z'\to z}\sup\{\frac {1}{n}\ln\big|p_n(z')\big|\colon\big\Arrowvert p_n\big\Arrowvert_\Delta\leq 1\}
$$
is the extremal function of Green. 

We introduce the following auxiliary sequence of plurisubharmonic functions
$$
u_k(z,w)=\frac {1}{n_k}\ln\big|q_{n_k}(z)\cdot w-p_{n_k}(z)\big|,\qquad (z,w)\in\mathbb C^2.
$$
For $(z,w)\in\mathbb C^2$ we have 
\begin{equation*}
\begin{split}
& \frac {1}{n_k}\ln\big|q_{n_k}(z)\cdot w-p_{n_k}(z)\big|\leq \frac {1}{n_k}\ln\Big(\big|q_{n_k}(z)\cdot w\big|+\big|p_{n_k}(z)\big|\Big)\\
& \leq\max\{\frac {1}{n_k}\ln 2\big|p_k(z)\big|,\,\,\frac {1}{n_k}\ln 2\big|q_{n_k}(z)\cdot w\big|\}\\
& =\max\{\frac {1}{n_k}\ln \big|p_k(z)\big|,\,\,\frac {1}{n_k}\ln \big|q_{n_k}(z)\big|+\frac {1}{n_k}\ln\big|w\big|\}+\frac {\ln 2}{n_k}.
\end{split}
\end{equation*}
From here we obtain the following estimate
$$
u_k(z,w)\leq \max\{V^\star(z,\Delta),\,\, V^\star(z,\Delta)+\frac {1}{n_k}\ln |w|\}+\frac {\ln 2}{n_k}.
$$
Consequently, the sequence of plurisubharmonic functions $u_k(z,w)$ is locally uniformly bounded from above.  Let 
 $$
 u(z,w)=\overline{\lim_{k\to\infty}} u_k(z,w).
 $$
The function $u(z,w)$ is also locally bounded from above,  
$$
u(z,w)\leq V^\star(z,\Delta).
$$
Let
$$
u^\star(z,w)=\underset {(z',w')\to (z,w)}{\overline\lim} u(z',w')
$$
denote the regularization of function $u(z,w)$.  The set 
$$
E=\{(z,w)\in\mathbb C^2\colon u(z,w)<u^\star(z,w)\}
$$
is pluripolar in $\mathbb C^2$ (see \cite{S1},\cite{S2}).

Let now $(z,w)\in\Gamma_f$ be a fixed point. (Note that $q_n(z)\neq 0$ for $z\in\Delta$). Then 
\begin{equation*}
\begin{split}
u(z,w)=\underset{k\to\infty}{\overline\lim}\ln\big|q_{n_k}(z)\big|^{\frac {1}{n_k}}\Big|w-\frac {p_{n_k}(z)}{q_{n_k}(z)}\Big|^{\frac {1}{n_k}} & \leq\underset{k\to\infty}{\overline\lim}\ln\alpha\big|q_{n_k}(z)\big|^{\frac {1}{n_k}}\\
& =\ln\alpha+\underset{k\to\infty}{\overline\lim}\ln\big|q_{n_k}(z)\big|^{\frac {1}{n_k}}.
\end{split}
\end{equation*}
 If  $(z,w)\in (\Delta\times\mathbb C)\setminus\Gamma_f$, then  
\begin{equation*}
\begin{split}
u(z,w) &=\underset{k\to\infty}{\overline\lim}\frac {1}{n_k}\ln\big|q_{n_k}(z)w-p_{n_k}(z)\big|\\ 
& = \underset{k\to\infty}{\overline\lim}\ln\big|q_{n_k}(z)\big|^{\frac {1}{n_k}}\Big|w-\frac {p_{n_k}(z)}{q_{n_k}(z)}\Big|^{\frac {1}{n_k}}=
\underset{k\to\infty}{\overline\lim}\ln\big|q_{n_k}(z)\big|^{\frac {1}{n_k}}.
\end{split}
\end{equation*}
It follows that if  
$$
\underset{k\to\infty}{\overline\lim}\big|q_{n_k}(z)\big|^{\frac {1}{n_k}}\neq 0
$$
at a point $z\in\Delta$, then $(z,f(z))$  belongs to the pluripolar set  $E$. Therefore, to complete the proof of the theorem it is enough to show that the set
$$
A=\Big\{z\in\Delta\colon\underset{k\to\infty}{\overline\lim}\big|q_{n_k}(z)\big|^{\frac {1}{n_k}}=0\Big\}=
\Big\{z\in\Delta\colon\underset{k\to\infty}{\lim}\big|q_{n_k}(z)\big|^{\frac {1}{n_k}}=0\Big\}
$$
is polar. Assume that $A$ is not polar, i.e., it has a positive logarithmic capacity, $\operatorname{cap}(A)>0$.  

We consider the following sequence of subharmonic functions 
$$
\vartheta^\star_k(z)=\underset{z'\to z}{\overline\lim}\vartheta_k(z'),\qquad z\in\mathbb C,
$$ 
where  
$$
\vartheta_k(z)=\sup_{m\geq k}\big|q_{n_m}(z)\big|^{\frac {1}{n_m}}.
$$
It is clear that the sequence $\vartheta^\star_k(z)$ is  locally uniformly bounded, $0\leq\vartheta^\star_k(z)\leq e^{V^\star(z,\Delta)}$ and is not monotonically increasing. In addition, $\vartheta^\star_k(z)\to 0$  on $A$, except for the polar set 
$$
F=\bigcup_{k=1}^{\infty}\{z\in\mathbb C\colon\,\vartheta_k(z)<\vartheta^\star_k(z)\}
$$
since by definition the sequence  $\vartheta_k(z)$ tends to zero on $A$.

Since $\vartheta^\star_k(z)$ is monotonic, for every $\epsilon$, $0<\epsilon<\operatorname{cap}(A)$ there exists an open set  $U_\epsilon$, $\operatorname{cap}(U_\epsilon)<\epsilon$ such that the sequence $\vartheta^\star_k(z)$ converges uniformly on the set $A_\epsilon=A\setminus U_\epsilon$ (see, for example, \cite{S1}). It follows that there exists a compact set $A_0\subset A_\epsilon$, $\operatorname{cap}(A_0)>0$,   such that the sequence of subharmonic functions $\vartheta^\star_k(z)$ converges uniformly to zero on the set $A_0$.  Consequently,  the sequence $\big|q_{n_k}(z)\big|^{1/{n_k}}$   also converges uniformly to zero on the  compact set $A_0$.
Now we use the so-called  $\tau$-capacity, introduced by A.Sadullaev (see \cite{S1}). We consider a polynomial  $T_n(z)$,  $\operatorname{deg}T_n(z)\leq n$, normalized by condition  $\Arrowvert T_n\Arrowvert_{U(0,r)}=1$, where  $U(0,r)=\{z\colon |z|<r\}\supset\Delta$ for which the norm $\Arrowvert T_n\Arrowvert_{A_0}$ is minimal among all such polynomials. Then the limit 
$$
\lim_{n\to\infty}\Arrowvert T_n\Arrowvert_{A_0}^{\frac {1}{n}}=\tau\big(A_0,U(0,r)\big),
$$
exists and, moreover
$$
\Arrowvert T_n\Arrowvert_{A_0}^{\frac {1}{n}}\geq\tau\big(A_0,U(0,r)\big)
$$
for every  $n\in N$. In fact, 
$$
\tau\big(A_0,U(0,r)\big)=\exp\{-\max_{z\in U(0,r)}V^\star(z,A_0)\}.
$$

It follows that
\begin{equation}\label{one}
\big\Arrowvert q_{n_k}\big\Arrowvert_{A_0}^{\frac {1}{n_k}}\geq\big\Arrowvert T_{n_k}\big\Arrowvert_{A_0}^{\frac {1}{n_k}}\geq\tau\big(A_0,U(0,r)\big)>0
\end{equation}
for 
$$
\Arrowvert q_{n_k}\Arrowvert_\Delta=1\qquad\text{and}\qquad\Arrowvert q_{n_k}\Arrowvert_{U(0,r)}\geq 1.
$$
It follows from inequality (\ref{one}) that the sequence $|q_{n_k}(z)|^{\frac {1}{n_k}}$  does not converge uniformly to zero on  the set $A_0$.  We arrived at a contradiction. Hence the set $A$  is polar.

Thus, the graph $\Gamma_f$ is divided into two parts, $\Gamma_f=\Gamma'_f\cup\Gamma''_f$, where
$$
\Gamma'_f=\{(z,f(z))\colon\, z\in\Delta\setminus A\}\qquad\text{and}\qquad\Gamma''_f=\{(z,f(z))\colon\, z\in A\}.
$$
The first set is a subset of a pluripolar set $E\subset\mathbb C^2$ and the second set is a subset of a pluripolar set  $A\times\mathbb C\subset\mathbb C^2$. Consequently, $\Gamma_f$  is a pluripolar in $\mathbb C^2$.   The proof of the theorem is complete.

\end{proof}

\noindent{\bf{Sadullaev's additions.}}  Actually, in this work the authors proved a stronger result. Let $K$ be arbitrary compact set in $\mathbb C^n$ and let $f(z)\in C(K)$. We say that $f(z)$  belongs to the class of Gonchar $R(K)$  if
$$
\underset {k\to\infty}{\underline{\lim}}{\rho_k}^{1/k}(f,K)<1,\qquad\text{where}\qquad\rho_k(f,K)=\inf_{\{r_k\}}\max_{z\in K}|f(z)-r_k(z)|
$$
and the  infimum is taken over all rational functions $r_k(z)$ of degree less than or equal to $k$. The following more general result holds:
if the function $f(z)$ belongs to the class of Gonchar $R(K)$, then its graph $\Gamma_f=\{(z,f(z))\in\mathbb C^{n+1}\colon z\in K\}$   is pluripolar in $\mathbb C^{n+1}$.
 
  
We note that pluripolarity of a graph $\Gamma_f$ is closely related to analyticity property of $f(z)$ or to properties of $f(z)$ related to analyticity.  Analyticity of continuous functions in domains $D\subset \mathbb C^n$ whose graphs are pluripolar in $\mathbb C^{n+1}$ was proved in (\cite{Shc}). An example of a function $f(z)$, which is holomorphic in the unit disk $U\subset\mathbb C$ and continuous in the closed disk $\overline U$, and whose graph $\Gamma_f$,
$$
\Gamma_f=\{(z,f(z))\in\mathbb C^2\colon\, z\in\overline U\},
$$
is not pluripolar is constructed in (\cite{LMP}).  These results suggest that the main reason behind pluripolarity of graphs is the analyticity of functions. 

In (\cite{S3}) an example of a lacunary series
$$
f(z)=\sum_{k=1}^{\infty} a_{n_k}z^{n_k}\quad (\underset{k\to\infty}{\overline\lim}\sqrt[n_k]{a_{n_k}}=1,\quad\frac {a_{n_k}}{a_{n_k+1}}\to 0\quad\text{as}\quad k\to\infty)
$$
whose graph $\Gamma_f$ is plurispased in every boundary point. It would be interesting to establish if the graph of a function $f(z)$,
$$
f(z)=\sum_{k=1}^{\infty}\frac {1}{k^{\ln k}}z^{k!},
$$
which is infinitely smooth up to the closure of its domain, is pluripolar in $\mathbb C^2$?
If we consider the series 
$$
f(z)=\sum_{k=1}^{\infty}\frac {1}{2^{(k-1)!}}z^{k!},
$$
then its graph $\Gamma_f$  will be pluripolar since $f(z)$ is quasianalytic in the sense of Gonchar.


\begin{thebibliography}{99}

\bibitem{G1}
Gonchar A. A., {\it{On best approximations by rational functions}}, Dokl. Akad. Nauk SSSR, 100 (1955), 205--208, (Russian). 


\bibitem{G2}
Gonchar A. A., {\it{Quasianalytic class of functions, connected with best approximated by rational functions}},
Izv. of the Academy of Sciences Armenia SSR. 1971, VI . No. 2-3,  148--159.

\bibitem{CLP}
Coman, D., Levenberg, N. and Poletsky, E. A., {\it{Quasianalyticity and Pluripolarity}},
arxiv:math/0402381.v1[math.CV] 23 feb, 2004.

\bibitem{T}
Timan, A. F., {\it{Theory of approximation of functions of a real variable}}, Pergaon Press, Macmillian, New York, 1963.

\bibitem{DF}
Diederich, K. and Fornass, J. E., {\it{A smooth curve which is not a pluripolar set}}, Duke Math.J. 49 (1982), 931--936.

\bibitem{S1}
Sadullaev, A., {\it{Rational approximation and pluripolar sets}}, Mat. Sb., 119:1 (1982), 96--118; English transl. in Math. USSR-Sb., 47 (1984).

\bibitem{S2}
Sadullaev, A., {\it{Plurisubharmonic functions}}, Sovrem. Probl. Mat. Fund. Naprav, 8, (VINITI, Moscow), 1985, 65--113 ; English transl. in Encyclopaedia Math. Sci., vol. 8 [Several complex variables II], (Springer-Verlag, Berlin), 1984.

\bibitem{S3}
Sadullaev, A., {\it{Plurisubharmonic measures and capacities on complex manifolds}}, Uspekhi Mat. Nauk, 36:4 (1981), 53--105; English transl. in Russian Math. Surveys, 36 (1981).

\bibitem{LMP}
Levenberg, N., Martin, G. and Poletskiy, E. A., {\it{Analytic disks and pluripolar sets}}, Indiana Univ. Math. J. 41 (1992), 515--532.

\bibitem{Shc}
Shcherbina, N., {\it{Pluripolar graphs are holomorphic}}, Acta Math., 194 (2005), 203--216.

\end{thebibliography}
\end{document}